\documentclass[conference,letterpaper
]{IEEEtran}

\usepackage[utf8]{inputenc} 
\usepackage[T1]{fontenc}
\usepackage{ifthen}
\usepackage{cite}
\usepackage[cmex10]{amsmath} 

\usepackage{amsmath,amssymb,amsthm,mathrsfs}
\usepackage{epsfig,epsf,subfigure,graphicx,graphics}
\usepackage{enumerate}
\usepackage{colonequals}
\usepackage{float}
\usepackage[english]{babel}
\usepackage{subfig}
\usepackage{enumitem}
\usepackage{bm}
\usepackage{cite}

\newcommand{\E}{\mathbb{E}}
\newcommand{\Var}{\mathrm{Var}}
\newcommand{\Cov}{\mathrm{Cov}}

\newcommand{\tr}{\operatorname{Tr}}

\newcommand{\R}{\mathbb{R}}

\newcommand{\LL}{L}
\newcommand{\UU}{R}

\newcommand{\dimension}{d}
\newcommand{\observations}{n}
\newcommand{\fixed}{M}
\newcommand{\fixedvec}{m_i}

\newcommand{\ltwo}{$L_2$ }

\newtheorem{theorem}{Theorem}
\newtheorem{lemma}[theorem]{Lemma}

\newtheorem{corollary}[theorem]{Corollary}
\newtheorem{definition}[theorem]{Definition}

\newtheorem{assumption}{Assumption}
\newtheorem{remark}[theorem]{Remark}

\begin{document}
\title{Linear Models are Most Favorable among Generalized Linear Models} 


\author{%
  \IEEEauthorblockN{Kuan-Yun Lee and Thomas~A.~Courtade}
  \IEEEauthorblockA{Department of Electrical Engineering and Computer Sciences\\
                    University of California, Berkeley\\ 
                    \{timkylee, courtade\}@berkeley.edu}
}


\maketitle





\begin{abstract}
We establish a nonasymptotic lower bound on the $L_2$ minimax risk for a class of generalized linear models.  It is further shown that the minimax risk for the  canonical linear model matches this lower bound up to a universal constant.  Therefore, the canonical linear model may be regarded as most favorable among the considered class of generalized linear models (in terms of minimax risk).  The proof makes use of an information-theoretic Bayesian Cram\'er-Rao bound for log-concave priors, established by Aras et al. (2019).

\end{abstract}

\section{Introduction and Main Results}
\label{sec:intro}

As their name suggests, generalized linear models (GLMs) are a flexible class of parametric statistical models that generalize the class of linear models relating a random  observation $X \in \mathbb{R}^{\observations}$ to a parameter $\theta\in \mathbb{R}^\dimension$ via the linear relation
\begin{align}
    X = M \theta + Z, \label{eq:LinearModel}
\end{align}
where   $M \in \mathbb{R}^{\observations\times \dimension}$ is a known (fixed)  design matrix, and $Z \in \mathbb{R}^{\observations}$ is a random noise vector. For a univariate GLM in canonical form with natural parameter $\eta \in \R$, the density of observation $X \in \R$ given $\eta$ is expressed as the exponential family
$$
f(x; \eta) = h(x) \exp\left( \frac{ \eta x  - \Phi(\eta)}{s(\sigma)} \right),  
$$
for known functions $h: \mathcal{X}\subseteq \R \to [0,\infty)$ (the \emph{base measure}), $\Phi: \mathbb{R} \to \mathbb{R}$ (the \emph{cumulant function}) and a scale parameter $s(\sigma)>0$. For this general class of models, the question of central importance is how well one can estimate $\eta$ from an observation $X \sim f(\cdot ; \eta)$, where $f(\cdot ; \eta)$ is understood to be a density on a probability space $(\mathcal{X} \subseteq \mathbb{R},\mathcal{F})$ with respect to a dominating $\sigma$-finite measure $\lambda$.  This class of models captures a wide variety of parametric  models such as binomial, Gaussian, Poisson, etc.  
As a specific example, we can take $\mathcal{X} = \{0,1,2, \ldots\}$ equipped with the counting measure $\lambda$.  For $h(x) = 1/x!$,  $\Phi(t) = e^{t}$ and $s(\sigma) = 1$, the observation  $X\sim f(\cdot; \eta)$  will  be $\mathrm{Poisson}(e^{\eta})$.

In this paper, we restrict our attention to multivariate GLMs of the form 
\begin{equation} \label{eq:glm}
    \begin{split}
        f(x; \theta) = \prod_{i=1}^\observations \left\{h(x_i) \exp\left(\frac{x_i \langle \fixedvec,\theta\rangle - \Phi(\langle \fixedvec,\theta\rangle)}{s(\sigma)}\right)\right\},
    \end{split}
\end{equation}
for a real parameter $\theta \in \mathbb{R}^\dimension$ and a fixed design matrix $M \in \mathbb{R}^{\observations\times \dimension}$, with rows given by the  vectors $\{m_i\}_{i=1}^n \subset \mathbb{R}^d$. 
In words, the above model assumes each $X_i$ is drawn from the same exponential family, with respective natural  parameters  $\langle \fixedvec,\theta\rangle$, $i=1,2, \ldots, n$. This captures the linear model \eqref{eq:LinearModel} in the usual case where the noise vector $Z$ is assumed to be standard normal on $\R^n$, but is also flexible enough to capture many other models of interest.

In terms of parameter estimation, a typical figure of merit is the constrained $L_2$ minimax risk, which corresponds to the worst-case $L_2$ estimation error, where $\theta$ is allowed to range over a constrained set $\Theta$. For our purposes, we take $\Theta$ equal to the Euclidean ball in $\mathbb{R}^d$, denoted $\mathbb{B}_2^\dimension(1):= \{v: v \in \mathbb{R}^\dimension, \|v\|_2^2 \leq 1\}$, which is a common choice in applications.   More precisely, we make the following definition. 
\begin{definition}
For the generalized linear model \eqref{eq:glm}, we define the associated minimax risk via
$$
R^*(h,\Phi, M,s(\sigma))  := \inf_{\hat{\theta}} \sup_{\theta \in \mathbb{B}_2^\dimension(1)} \mathbb{E}\|\theta - \hat{\theta}\|_2^2, 
$$
where the expectation is over $X\sim f(\cdot; \theta)$, and the infimum is over all $\mathbb{R}^d$-valued estimators $\hat{\theta}$ (i.e., measurable functions of the observation  $X$).  
\end{definition}

Before we state our main results, we make the following assumption throughout:
\vskip1ex

\begin{assumption}\label{assumption:CumulantFn}
We assume the cumulant function $\Phi: \mathbb{R}\to \mathbb{R}$ in \eqref{eq:glm} is twice-differentiable, with second derivative uniformly bounded as $\Phi'' \leq L$, for some $L>0$. 
\end{assumption}

\begin{remark}
This assumption is standard in the literature on minimax estimation for GLMs, and is equivalent to the map $\theta \longmapsto \E_{X\sim f(\cdot; \theta)}[X]$ being $L$-Lipschitz. See, for example, \cite{abramovich2016model, muller2005generalized, negahban2012unified, loh2015regularized}. 
\end{remark}

Our first main result is a general lower bound on the minimax risk for the class of GLMs introduced above.
\begin{theorem}\label{thm:minimaxBound}
The $L_2$ minimax risk for the class of models \eqref{eq:glm} is lower bounded according to
\begin{equation} \label{eq:minimaxBound}
R^*(h,\Phi, M,s(\sigma))   \gtrsim \min\left( \frac{s(\sigma)}{L} \, \tr\left((\fixed^\top \fixed)^{-1}\right) , 1\right),
\end{equation}
where $\gtrsim$ denotes inequality, up to a universal constant. 
\end{theorem}
\begin{remark}
In case $\fixed^\top \fixed$ is not invertible, we adopt the convention that $\tr\left((\fixed^\top \fixed)^{-1}\right)= +\infty$.  This situation occurs when $M$ is not full rank, in which case $\theta$ is not identifiable in the null space of $M$ and constant error is unavoidable.
\end{remark}
\begin{remark}
In fact, with minor modification, Theorem \ref{thm:minimaxBound} holds for the more general class of GLMs with observations drawn from densities of the form
\begin{align*}
        f(x; \theta) = \prod_{i=1}^\observations \left\{h_i(x_i) \exp\left(\frac{x_i \langle \fixedvec,\theta\rangle - \Phi_i(\langle \fixedvec,\theta\rangle)}{s_i(\sigma)}\right)\right\}.
\end{align*}
See Section \ref{subsec:Rmks} for further remarks. 
\end{remark}

\begin{remark} \label{remark:constant}
Since minimax risk is generally characterized modulo universal constants, the statement \eqref{eq:minimaxBound} in terms of $\gtrsim$ is sufficient for our purposes.  However, a careful analysis of our arguments reveals that $\gtrsim$ can be replaced with $\geq$ at the expense of including a modest constant in the RHS of \eqref{eq:minimaxBound} (e.g.,  $ 1 / (\pi e^3)$).
\end{remark}


Most interestingly, the minimax bound \eqref{eq:minimaxBound} holds uniformly over the class of GLMs given by \eqref{eq:glm}, and is of the correct order for the canonical linear model \eqref{eq:LinearModel}. Indeed, under the linear  model $X = L \fixed\theta + Z$, where $Z$ is standard Gaussian with covariance $\sigma^2 L \cdot I$ and  the design matrix $\fixed$ is full rank, the maximum likelihood estimator (MLE) estimator $\hat{\theta}_{\text{MLE}}$ is given by
\begin{equation*}
    \hat{\theta}_{\text{MLE}} =  L^{-1} (M^\top M)^{-1} M^\top X.
\end{equation*}
One can explicitly calculate the \ltwo error as 
\begin{align}
        \E \|\theta -  \hat{\theta}_{\text{MLE}}\|_2^2 &= \E \| \theta -  L^{-1}(M^\top M)^{-1} M^\top X\|_2^2 \notag\\
        &= \frac{1}{L^2} \E \| (M^\top M)^{-1} M^\top Z\|_2^2 \notag\\
        &= \frac{\sigma^2}{L} \tr((M^\top M)^{-1}).
        \label{eq:linearregression}
\end{align}
The linear  model in this case corresponds to $h(x) = e^{-x^2/(2 L \sigma^2)}$,  $s(\sigma) = \sigma^2$, and $\Phi(t) = L t^2/2$ in \eqref{eq:glm}.

Comparing \eqref{eq:linearregression} to Theorem \ref{thm:minimaxBound}, we find that our minimax lower bound is achieved (up to a universal constant) for linear models of the above form. To summarize, we have the following:
\begin{corollary}
Fix a design matrix $M$, scale parameter $s(\sigma)$ and $L>0$.  Among those generalized linear models in \eqref{eq:glm} with $\Phi'' \leq L$, linear models are most favorable in terms of minimax risk.    More precisely, among this class of models, 
$$
R^*(h,\Phi, M,s(\sigma))  \gtrsim   R^*( e^{-(\cdot)^2/(2 L s(\sigma))}  ,  (\cdot)^2 L/2, M,s(\sigma)).
$$
\end{corollary}
Roughly speaking, the above asserts that linear models are most favorable among a broad class of GLMs, giving this paper its name.

\subsection{Related Work}

Perhaps  most closely related to our  work is that of Abramovich and Grinshtein \cite{abramovich2016model}, albeit for a slightly different setup. In particular, Abramovich and Grinshtein  provide minimax lower bounds for the Kullback-Leibler divergence between the vector $M\theta$ and any estimator $\widehat{\fixed \theta}$ under a $k$-sparse setting $\|\theta\|_0 \leq k $,  with the parameter $\theta$ constrained to have at most $k$ non-zero entries. 
When the cumulant function $\Phi$ is strongly convex with $0 < R\leq \Phi'' \leq L$ for some fixed constants $\UU,L$, we can adapt the arguments of \cite{abramovich2016model} to obtain the following $L_2$ minimax lower bound
\begin{align*} 
        \inf_{\widehat{\fixed \theta}} \sup_{\theta\in \mathbb{B}^\dimension_2(1)} \|M\theta  - \widehat{\fixed \theta}\|_2^2 \gtrsim   \frac{ d s(\sigma)  \UU}{\LL^2}  \cdot  \frac{\lambda_{\min}(\fixed^\top \fixed)}{\lambda_{\max}(\fixed^\top \fixed)} ,
\end{align*}
where $M$ is assumed to be full rank and $\lambda_{\min}$ and $\lambda_{\max}$ denote smallest and largest eigenvalues, respectively.    The minimax lower bound for estimating $\fixed\theta$ is not directly comparable to our result, where the goal is estimation of  $\theta$. Nevertheless, using the operator norm inequality $\|\fixed(\theta - \hat{\theta})\|_2^2 \leq \lambda_{\max} (\fixed^\top \fixed) \|\theta - \hat{\theta}\|_2^2$, we may conclude
\begin{align*}
        \inf_{\hat{\theta}} \sup_{\theta\in \mathbb{B}^\dimension_2(1)} \|\theta  - \hat{\theta }\|_2^2 \gtrsim \frac{ d s(\sigma) R}{L^2} \cdot   \frac{\lambda_{\min}(\fixed^\top \fixed)}{\lambda_{\max}^2(\fixed^\top \fixed)} .
\end{align*}
A direct computation shows that  \eqref{eq:minimaxBound} is sharper than the above $L_2$ minimax estimate since
\begin{equation*}
    \begin{split}
    \frac{d \, \lambda_{\min}(\fixed^\top \fixed)}{\lambda_{\max}^2(\fixed^\top \fixed)}  &\leq \frac{\dimension}{\lambda_{\max}(\fixed^\top \fixed)} 
    \leq \tr\left(\left(\fixed^\top \fixed \right)^{-1}\right).
    \end{split}
\end{equation*}

As for a general theory, apart from the gaussian linear model, the minimax estimator for the GLM does not have a closed form, but the Maximum Likelihood Estimator (MLE) can be  approximated by  iterative weighted linear regression \cite{nelder1972generalized}. A variety of estimators such as aggregate estimators \cite{rigollet2012kullback}, robust estimators \cite{cantoni2001robust} and GLM with Lasso \cite{van2008high} have been proposed to solve different settings of the GLM. We refer interested readers to \cite{mccullagh2019generalized} for the  theory of  GLMs. 

Another line of related work explores models with stochastic design matrix $\fixed$. Duchi, Jordan and Wainwright \cite{duchi2018minimax} consider  inference of a parameter $\theta$ under privacy constraints. Negahban et al. \cite{negahban2012unified} and Loh et al. \cite{loh2015regularized} provide consistency and convergence rates for M-estimators in GLMs with low-dimensional structure under high-dimensional scaling. 

Separate from the minimax problems considered here, model selection is another line of popular work. Model selection in linear regression dates back to the seventies and has regained popularity over the past decade, due to the increase in need of data exploration for high dimensional data; see \cite{akaike1998information, birge2007minimal, verzelen2012minimax} and many other works for the history. More recently, tools in model selection for linear regression have been adapted for the GLM; see \cite{abramovich2016model} for a brief discussion. 




\subsection{Organization}
The remainder of this paper is organized as follows. Preliminaries for the derivation of our minimax lower bounds are introduced in Section \ref{subsec:tools}. The proof of Theorem \ref{thm:minimaxBound} is given in Section \ref{subsec:minimax}, with further remarks in Section \ref{subsec:Rmks}. 

\section{Derivation of Minimax Bound for the GLM}
\label{sec:glm}

The following notation is used throughout:   upper-case letters (e.g., $X$, $Y$) denote random variables or matrices, and lower-case letters (e.g., $x, y$) denote realizations of random variables  or vectors. We use  subscript notation $v_i$ to denote the $i$-th component of a vector $v = (v_1, v_2, \ldots, v_d)$, and we define the leave-one-out vector $v^{(j)}:= (v_1,\ldots, v_{j-1}, v_{j+1},\ldots, v_d)$.

\subsection{Preliminaries} \label{subsec:tools}
In the general framework of parametric statistics, let $(\mathcal{X}, \mathcal{F}, P_{\theta}; \theta \in \mathbb{R}^d)$ be a dominated family of probability  measures on a measurable space $(\mathcal{X}, \mathcal{F})$ with dominating $\sigma$-finite measure $\lambda$. To each $P_{\theta}$, we associate a density $f(\cdot;\theta)$ with respect to $\lambda$  according to 
\begin{equation}
    dP_{\theta}(x) = f(x;\theta) d\lambda(x).
\end{equation}
Assuming the maps $\theta \longmapsto f(x;\theta)$, $x\in \mathcal{X}$, are differentiable, the Fisher information matrix associated to observation $X \sim f(\cdot;\theta)$ and parameter $\theta \in \mathbb{R}^\dimension$ is defined as the matrix-valued map $\theta \longmapsto \mathcal{I}_X(\theta)$ with components
\begin{equation*}
    [\mathcal{I}_X(\theta)]_{ij} = \E \left[ \frac{\partial \log f(X;\theta)}{\partial \theta_i}    \frac{\partial \log f(X;\theta)}{\partial \theta_j}  \right], ~~~\theta \in \R^d.
\end{equation*}
Here and throughout, $\log$ denotes the natural logarithm.
The following regularity assumption is standard when dealing with Fisher information. 
\begin{assumption} \label{assumption:assumption}
The densities $f(\cdot;\theta)$ are sufficiently regular to permit the following exchange of integration and differentiation: 
    \begin{align} 
            \int_\mathcal{X} \nabla_\theta f(x;\theta) d\lambda(x) = 0;  ~~~\theta \in \R^d. \label{eq:assumedRegularity}
    \end{align}
    Here,  $\nabla_\theta$ denotes the gradient with respect to $\theta$.
\end{assumption}

While the Fisher information is one notion of   \emph{information} that an observation $X\sim f(\cdot; \theta)$ reveals about the unknown parameter $\theta$, it also makes sense to consider the usual mutual information $I(X;\theta)$ under the further assumption that $\theta$ is distributed according to a known prior distribution $\pi$ (a probability measure on $\R^d$).  Recent results by the authors together with Aras and Pananjady  establish a quantitative relation between these two notions of information \cite{alpc}.  To state the result precisely, recall that a probability measure $d\mu = e^{-V}dx$ on $\R^d$ is said to be log-concave if the potential $V: \R^d\to \R$ is a convex function.  
\begin{lemma}[\!\!{\cite[Theorem 2]{alpc}}] \label{lem:alpc}
Let $\theta \sim \pi$, where $\pi$ is log-concave on $\R^d$, and given $\theta$ let $X\sim f(\cdot; \theta)$.  If     Assumption \ref{assumption:assumption} holds, then 
    \begin{align}
                I(X;\theta) \leq \dimension\cdot \phi\left(
                \frac{\tr\left(\Cov(\theta)\right)\cdot \tr\left(\E \, \mathcal{I}_X(\theta) \right)}{d^2} \right), \label{eq:alpc-2}
\end{align}
where
\begin{align*}
        \phi(x) := 
        \begin{cases}
        \sqrt{x} & \mbox{if $0\leq x \leq 1$}\\
        1 + \frac{1}{2}\log x & \mbox{if $x\geq 1$.}
        \end{cases}
\end{align*}
\end{lemma}
As discussed extensively in \cite{alpc}, the above result is  related to the van Trees inequality \cite{bcrb,van2004detection}, and its entropic improvement due to Efroimovich \cite{efro}.  The crucial feature of \eqref{eq:alpc-2} compared to these other results is that it does not depend on the (information theorist's version of) Fisher information of the prior $\pi$, commonly denoted $\mathcal{J}(\pi)$.  This is what is gained via the assumption of log-concavity, and is important for our analysis where we introduce (log-concave) priors with arbitrarily large Fisher information.

\subsection{Proof of Theorem \ref{thm:minimaxBound}} \label{subsec:minimax}



Recall that the design matrix $M$ has as its rows $\{m_i\}_{i=1}^n \subset \mathbb{R}^d$.  Writing the matrix $\fixed$ in terms of its SVD $\fixed = U\Sigma V^\top$ and defining $u_i$ as the $i$-th column of the matrix $U^{\top}$, we have
\begin{align}
    \langle \fixedvec, \theta \rangle = \langle \underbrace{ \Sigma u_i}_\text{$\Bar{m}_i$} ,  \underbrace{V^\top \theta}_\text{$\Bar{\theta}$}\rangle  = \langle \Bar{m}_i, \Bar{\theta} \rangle,  \label{eq:rotation}
\end{align}
where we defined the variables $\Bar{m}_i:=  \Sigma u_i$ and $\Bar{\theta} := V^\top \theta$. Since $V$ is an orthogonal matrix by definition, it follows by rotation invariance of the $L_2$ ball $\mathbb{B}_2^d(1)$ that the estimation problem can be equivalently formulated under the reparametrization $(\theta, M) \longrightarrow (\bar\theta, \bar M)$, where  $\bar M := MV = U\Sigma$.  More specifically, the minimax risk for $\theta$ over the set of estimators for estimating $\theta\in \mathbb{B}_2^d(1)$  is equal to the minimax risk for estimating $\Bar{\theta} \in \mathbb{B}_2^d(1)$. More precisely,  
\begin{align*}
    \inf_{\hat{\theta}} \sup_{\theta \in \mathbb{B}_2^\dimension(1)} \E\|\theta - \hat{\theta}\|_2^2 = \inf_{ \hat{\Bar{\theta}}} \sup_{\Bar{\theta} \in \mathbb{B}_2^\dimension(1)} \E\|\Bar{\theta} - \hat{\Bar{\theta}}\|_2^2. 
\end{align*}
As a result, we may assume without loss of generality that $\fixed^\top \fixed$ is a diagonal matrix. 

By definition, minimax risk is lower bounded by the Bayes risk, when $\theta$ is assumed to be distributed according to a prior $\pi$, defined on the \ltwo ball $\mathbb{B}_2^d(1)$. Hence, our task is to judiciously select a  prior $\pi$ that yields the desired lower bound.  Toward this end, we will let $\pi$ be the uniform measure on the rectangle $\prod_{i=1}^d [-\epsilon_i/2, \epsilon_i/2]$ for values $(\epsilon_i)_{i = 1, 2, \ldots, d}$ to be determined below satisfying 
\begin{align}
\sum_{i=1}^d \epsilon_i^2\leq 4.  \label{eq:epsilon-requirement}
\end{align}
In other words, our construction implies $\theta$ has independent components, with the $i$-th coordinate $\theta_i$ uniform on the interval $[-\epsilon_i/2, \epsilon_i/2]$.  The interval lengths will, in general,  be chosen to exploit the structure of the design matrix $M$.

%

We now describe our construction of the sequence $(\epsilon_i)_{i=1,2, \ldots, d}$. We start with the simple case, in which the matrix $M$  does not have full (column) rank.  In this case,  there exists an eigenvalue  $\lambda_k(\fixed^\top \fixed) = 0$.  For this index $k$, we set $\epsilon_i = 2\delta_{ik}$, $i=1,2, \ldots, d$, where $\delta_{ij}$ is the Kronecker delta function.   Now, we may bound
\begin{align}
        \E \|\theta - \hat{\theta}\|_2^2 &\geq \Var(\theta_k - \hat{\theta}_k) \notag \\ 
        &\stackrel{(a)}{\geq} \frac{1}{2\pi e}   e^{2h(\theta_k - \hat{\theta}_k)}  \stackrel{(b)}{\geq} \frac{1}{2\pi e}   e^{2h(\theta_k | \hat{\theta}_k )} \notag\\
        & = \frac{1}{2\pi e}   e^{2h(\theta_k) - 2I(\hat{\theta}_k;\theta_k)} \stackrel{(c)}{\geq} \frac{1}{2\pi e}   e^{2h(\theta_k) - 2I(X;\theta_k)} \notag \\
        &\stackrel{(d)}{=} \frac{2}{\pi e} \notag
\end{align}
where (a) follows from the max-entropy property of gaussians; (b) follows since conditioning reduces entropy: $h(\theta_k - \hat{\theta}_k) \geq h(\theta_k - \hat{\theta}_k|\hat{\theta}_k) = h(\theta_k|\hat{\theta}_k)$;  (c) follows from the data processing inequality since $\theta_k\to X \to \hat{\theta}_k$ forms a Markov chain; and (d) follows since $\theta_k\sim \operatorname{Unif}(-1,1)$ and $I(X;\theta_k)=0$, since $\pi$ is supported in the kernel of $M$ by construction.

Having shown the minimax risk is lower bounded by a constant when $M$ does not have full (column) rank, we assume henceforth that $\fixed$ has full   rank.  

Note that under our assumptions, the pair $(X,\theta)$ has a joint distribution, and therefore so does the pair $(X,\theta_i)$.  Consistent with the previously introduced notation, we write $\mathcal{I}_X(\theta_i)$ to denote the Fisher information of $X$ drawn according to the conditional law of $X$ given $\theta_i$.  With this notation in hand, the next lemma provides a comparison  between the expected Fisher information conditioned on a single component $\theta_i$ of the parameter $\theta$ and the $i$-th diagonal entry of the expected Fisher information matrix conditioned for parameter $\theta$.
\begin{lemma}\label{lem:independent}
    When the components of parameter $\theta \sim \pi$, $\theta \in \R^d$ are independent and $X\sim f(\cdot;\theta)$ is generated by the GLM \eqref{eq:glm}, we have
    \begin{align*}
            \E  \left[\mathcal{I}_X(\theta)\right]_{ii} \geq \E \, \mathcal{I}_X(\theta_i) ~~~i = 1,2,\ldots, d.
    \end{align*}
\end{lemma}
\begin{proof}
The desired estimate is obtained by observing
\begin{align*} 
        \E [\mathcal{I}_X(\theta)]_{ii} &= \E \left[ \frac{\left(\frac{\partial}{\partial \theta_i} f(X;\theta)\right)^2}{f(X;\theta)^2}  \right]\\
        &\stackrel{(a)}{\geq} \E \left[ \frac{\left(\E\left[\left. \frac{\partial}{\partial \theta_i} f(X;\theta) \right|\theta_i, X\right]\right)^2}{ \left(\E\left[\left. f(X;\theta) \right|\theta_i, X\right]\right)^2} \right] \\
        &\stackrel{(b)}{=} \E \left[ \frac{\left(\frac{\partial}{\partial \theta_i}\E\left[\left.  f(X;\theta) \right|\theta_i, X\right]\right)^2}{ \left(\E\left[\left. f(X;\theta) \right|\theta_i, X\right]\right)^2}\right]= \E \, \mathcal{I}_X(\theta_i).
\end{align*}
In the above, (a) is due to Cauchy-Schwarz.  Indeed, let $\pi_i$ and $\pi^{(i)}$ denote the marginal laws of $\theta_i$ and $\theta^{(i)}$, respectively.  Using independence of $\theta_i$ and $\theta^{(i)}$, note that 
\begin{align*}
&\E \left[   \frac{\left(\frac{\partial}{\partial \theta_i} f(X;\theta)\right)^2}{f(X;\theta)^2}   \right] \\
&=\int_{\mathbb{R}}\int_{\mathcal{X}} \int_{\mathbb{R}^{d-1}}  \frac{\left(\frac{\partial}{\partial \theta_i} f(x;\theta)\right)^2}{f(x;\theta)}    d\pi^{(i)}(\theta^{(i)}) d\lambda(x) d\pi_i(\theta_i)\\
&\geq \int_{\mathbb{R}}\int_{\mathcal{X}}   \frac{ \left( \int_{\mathbb{R}^{d-1}} \frac{\partial}{\partial \theta_i} f(x;\theta) d\pi^{(i)}(\theta^{(i)}) \right)^2 }{ \int_{\mathbb{R}^{d-1}}  f(x;\theta) d\pi^{(i)}(\theta^{(i)})   }    d\lambda(x) d\pi_i(\theta_i)\\
&=\E \left[ \frac{\left(\E\left[\left. \frac{\partial}{\partial \theta_i} f(X;\theta) \right|\theta_i, X\right]\right)^2}{ \left(\E\left[\left. f(X;\theta) \right|\theta_i, X\right]\right)^2}\right],
\end{align*}
where the last line follows since  
$$x \longmapsto \E\left[\left.  f(X;\theta) \right|\theta_i, X=x\right] = \int_{\mathbb{R}^{d-1}}  f(x;\theta) d\pi^{(i)}(\theta^{(i)})$$
 is the density (w.r.t. $\lambda$) of $X$ given $\theta_i$. 
 
Equality (b) follows from independence between $\theta_i$ and $\theta^{(i)}$ and the Leibniz integral rule.  Application of the latter can be justified by the assumed regularity of $\Phi$ and compactness of $\mathbb{B}_2^d(1)$.
\end{proof}

Next, fix $\epsilon_i>0$.  Since $\theta_i \sim \operatorname{Unif}(-\epsilon_i/2, \epsilon_i/2)$ has log-concave distribution, and the GLM \eqref{eq:glm} satisfies Assumption \ref{assumption:assumption} (a consequence of Assumption \ref{assumption:CumulantFn} and the Leibniz integral rule, justified by regularity of $\Phi$), we can apply Lemmas \ref{lem:alpc} and  \ref{lem:independent}  to conclude 
\begin{align} 
        e^{2h(\theta_i|\hat{\theta}_i)} 
        &\geq e^{2h(\theta_i) - 2I(X;\theta_i)} \notag \\
        &\geq e^{2h(\theta_i) - 2\phi\left(\Var(\theta_i) \cdot \E \, \mathcal{I}_X(\theta_i)\right)} \notag \\ 
        &\geq \epsilon_i^2 e^{-2\phi\left( \frac{\epsilon_i^2}{12}   \E \, \left[\mathcal{I}_X(\theta))\right]_{ii}\right)}.\label{eq:single-component}
\end{align}
Note that the last inequality used the identities $\Var(\theta_i) = \frac{\epsilon_i^2}{12}$ and $h(\theta_i) = \log(\epsilon_i)$, holding by construction. 

Next, recall the following well-known identities associated with exponential families of the form we consider. 

\begin{lemma}[\!\!{\cite[Page 29]{mccullagh2019generalized}}] \label{lem:GLM-lem}
    Fix $m$ and $\theta$, and consider a density  $f(x;\theta) = h(x)\exp\left(\frac{x \langle m, \theta \rangle - \Phi(\langle m, \theta \rangle )}{s(\sigma)}\right)$ with respect to $\lambda$.  A random observation $X\sim f(\cdot;\theta)$ has mean $\Phi'(\langle m,\theta \rangle)$ and variance $ s(\sigma) \cdot \Phi''(\langle m, \theta \rangle)$. 
\end{lemma}

Combining the above with our assumption that $\Phi''\leq L$, we have for any $\theta\in \mathbb{R}^d$,
\begin{align} 
        [\mathcal{I}_X(\theta)]_{ii} 
        &= \E_{X\sim f(\cdot; \theta)} \left(\frac{\partial}{\partial \theta_i} \log f(X;\theta)\right)^2  \notag \\
        &= \frac{1}{s^2(\sigma)} \E_{X\sim f(\cdot; \theta)} \left(\sum_{j=1}^\observations M_{ji}\left(X_j - \Phi'(\langle m_j, \theta \rangle)\right)\right)^2 \notag \\
        &= \frac{1}{s^2(\sigma)}  \sum_{j=1}^\observations \left(M_{ji}^2 \,  \Var(X_j)\right) \notag \\
        &\leq  \frac{1}{s(\sigma)}  \sum_{j=1}^\observations \left(M_{ji}^2 \,  L\right) \notag\\
        &= \frac{\LL}{s(\sigma)} [\fixed^\top \fixed]_{ii}. \label{eq:fisher-information-value}
\end{align}

 
Putting \eqref{eq:single-component} and \eqref{eq:fisher-information-value} together, we conclude for any choice of $\epsilon_i>0$,
\begin{align}
        e^{2h(\theta_i|\hat{\theta}_i)} \geq \epsilon_i^{2} \exp\left[-2\phi\left(\frac{\epsilon_i^2}{12}  \frac{\LL}{s(\sigma)}[\fixed^\top \fixed]_{ii}\right)\right]. \label{eq:single}
\end{align}
In case $\epsilon_i=0$, we have the trivial equality $e^{2h(\theta_i|\hat{\theta}_i)}=0$, which is consistent with the RHS of \eqref{eq:single} evaluated at $\epsilon_i=0$.  Hence, the estimate \eqref{eq:single} holds for all $\epsilon_i\geq 0$.  

Summing \eqref{eq:single} from $i=1,2, \ldots, d$,  for parameter $\theta \sim \pi = \prod_{i=1}^d \operatorname{Unif}(-\epsilon_i/2, \epsilon_i/2)$ and any measurable function $\hat\theta$ of $X\sim f(\cdot; \theta)$, we have the following lower bound on the Bayesian  $L_2$ risk,
\begin{align}  
            \E \|\theta - \hat{\theta}\|_2^2 &\geq  \sum_{i=1}^d \Var(\theta_i - \hat\theta_i) \notag\\
            &\geq \frac{1}{2 \pi e}\sum_{i=1}^\dimension  e^{2h(\theta_i|\hat{\theta}_i)} \notag \\
            &\geq \frac{1}{2 \pi e} \sum_{i=1}^\dimension  \epsilon_i^{2} \exp\left[-2\phi\left(\frac{\epsilon_i^2}{12}  \frac{\LL}{s(\sigma)}[\fixed^\top \fixed]_{ii}\right)\right].  \label{eq:glm-ineqxx}
\end{align}

It remains to choose an appropriate sequence $(\epsilon_i)_{i=1,2,\ldots, d}$ to obtain the desired lower bound.  Toward this end, we consider two cases:

\smallskip
\noindent{\bf Case 1:} $\tr((\fixed^\top \fixed)^{-1}) \leq \frac{1}{3}   \frac{\LL}{s(\sigma)}$. 

In this case, we choose $\epsilon_i^2 = 12 \frac{s(\sigma)}{\LL} \left([\fixed^\top \fixed]_{ii}\right)^{-1} 
$ for  $i = 1,2,\ldots,d$. Note that by our assumption that $M^{\top} M$ is diagonal,
\begin{align*}
        \sum_{i=1}^\dimension \epsilon_i^2  
        &= 12 \frac{s(\sigma)}{\LL}\tr((\fixed^\top \fixed)^{-1}) \leq 4,
\end{align*}
so that \eqref{eq:epsilon-requirement} is satisfied.  By an application of  \eqref{eq:glm-ineqxx}, we have
\begin{align*}
        \E \|\theta - \hat{\theta}\|_2^2  &\gtrsim \sum_{i=1}^\dimension  \epsilon_i^{2}  \exp\left[-2\phi\left(\frac{\epsilon_i^2}{12}\frac{\LL}{s(\sigma)}[\fixed^\top \fixed]_{ii}\right)\right] \\
        &= \frac{12}{e^2}  \frac{s(\sigma)}{\LL} \sum_{i=1}^\dimension \frac{1}{[\fixed^\top \fixed]_{ii}} \\
        &\gtrsim   \frac{s(\sigma)}{\LL} \tr((\fixed^\top \fixed)^{-1}).
\end{align*}

\smallskip
\noindent{\bf Case 2:} $\tr((\fixed^\top \fixed)^{-1}) > \frac{1}{3}   \frac{\LL}{s(\sigma)}$.

This case is the more difficult of the two.  We shall make use of the following technical Lemma.  
\begin{lemma}\label{lem:hard-lemma}
    Let $(a_i)_{i = 1,2,\ldots, d}$ be any positive sequence satisfying $\sum_{i=1}^\dimension a_i^{-1} > 4$. Then, there exists a non-negative sequence $(\epsilon_i)_{i = 1, 2, \ldots, d}$ such that $\sum_{i=1}^\dimension \epsilon_i^2 \leq 4$ and $\sum_{i=1}^\dimension \epsilon_i^2 e^{-2\phi(\epsilon_i^2 a_i)} \geq 2 e^{-2}$. 
\end{lemma}
\begin{proof}
    Without loss of generality, assume that $a_1 \geq a_2 \geq \cdots \geq a_d > 0$.  If $a_1 \leq 1/4$, then taking $(\epsilon_1,\epsilon_2,\ldots, \epsilon_d) = \left(2,0,0,\ldots, 0\right)$, and noticing that $\phi$ is an increasing function, we conclude
\begin{equation*}
    \begin{split}
        \sum_{i=1}^\dimension \epsilon_i^2 e^{-2\phi(\epsilon_i^2 a_1)} = 4 e^{-2\phi(4 a_1)} \geq 4 e^{-2\phi(1)} > 2 e^{-2}.
    \end{split}
\end{equation*}
Now, in the following we assume that $a_1> 1/4$.  Let $t$ denote the largest integer $k\in \{1,2,\ldots, d\}$ satisfying $\sum_{i=1}^k a_i^{-1} \leq 4$. By the assumption that $\sum_{i=1}^d a_i^{-1} > 4$, we know that there always exists such a $t$, and $t$ will satisfy $t < d$. We set
\begin{align}
    \epsilon_i = \left\{
    \begin{array}{ll}
        a_i^{-1/2} & \text{if $1\leq i \leq t$} \\
        0 & \text{otherwise}
    \end{array}
    \right. ~~~i = 1,2,\ldots,d. \label{eq:design}
\end{align}
By definition, $\sum_{i=1}^d \epsilon_i^2 = \sum_{i=1}^t a_i^{-1} \leq 4$ satisfies \eqref{eq:epsilon-requirement}. This procedure results in
\begin{equation*}
    \begin{split}
        \sum_{i=1}^\dimension \epsilon_i^2 e^{-2\phi(\epsilon_i^2 a_i)} = e^{-2} \sum_{i=1}^{t} \frac{1}{a_i}.
    \end{split}
\end{equation*}

If $\sum_{i=1}^{t} a_i^{-1} \geq 2$, we can immediately see from the above and \eqref{eq:design} that  $\sum_{i=1}^\dimension \epsilon_i^2 e^{-2\phi(\epsilon_i^2 a_i)}  \geq 2 e^{-2}$. 

On the other hand, if $\sum_{i=1}^{t} a_i^{-1} < 2$, this implies that $ a_{t+1}^{-1} \geq 2$. In this case, we simply take $\epsilon_{t+1} = 2$,  and take $\epsilon_i = 0$ for $i\neq t+1$. With this choice, we have
\begin{equation*}
    \begin{split}
        \sum_{i=1}^\dimension \epsilon_i^2 e^{-2\phi(\epsilon_i^2 a_i)} = 4 e^{-2\phi(4 a_{t+1})} \geq 4 e^{-2\phi(2)} = 2 e^{-2}.
    \end{split}
\end{equation*}
The above discussion concludes the proof of Lemma \ref{lem:hard-lemma}.
\end{proof}
By considering the values $a_i = \frac{\LL}{12 s(\sigma)}   [\fixed^\top \fixed]_{ii}$, Lemma \ref{lem:hard-lemma} ensures the existence of $(\epsilon_i)_{i=1,2,\ldots, d}$ satisfying \eqref{eq:epsilon-requirement} and, together with \eqref{eq:glm-ineqxx}, gives $
 \E \|\theta - \hat{\theta}\|_2^2 \gtrsim 1$. 
This completes the proof of Theorem \ref{thm:minimaxBound}.

\subsection{Remarks}\label{subsec:Rmks}
A few remarks are in order.  First, we note that the argument in the previous subsection actually yields the stronger entropic inequality, 
    $$
        \inf_{\hat{\theta}} \sup_{\theta\sim \pi} \sum_{i=1}^\dimension e^{2 h(\theta_i|\hat{\theta}_i)} \gtrsim  \min\left( \frac{s(\sigma)}{\LL}  \tr\left((\fixed^\top \fixed)^{-1}\right) , 1\right)
    $$ 
    which improves  Theorem \ref{thm:minimaxBound} (seen by the max-entropy property of gaussians). Here, the supremum is taken over all distributions $\pi$ supported on the \ltwo ball $\mathbb{B}_2^d(1)$.
    

Second, we remark that our analysis is flexible enough for generalizations to other forms of the GLM. For example, consider observation $X$ drawn from the density
\begin{align*}
        f(x; \theta) = \prod_{i=1}^\observations \left\{h_i(x_i) \exp\left(\frac{x_i \langle \fixedvec,\theta\rangle - \Phi_i(\langle \fixedvec,\theta\rangle)}{s_i(\sigma)}\right)\right\}.
\end{align*}
Suppose Assumption \ref{assumption:CumulantFn} holds for each cumulant function $\Phi_i$ (i.e., $\Phi_i''\leq L$ for each $i=1,\dots, n$). Then, a slight modification in \eqref{eq:fisher-information-value} yields
$$
    [\mathcal{I}_X(\theta)]_{ii} \leq \frac{L}{s^*(\sigma)} [M^\top M]_{ii}
$$
where $s^*(\sigma) = \min_{i=1,2,\ldots,n} s_i(\sigma)$. Following \eqref{eq:glm-ineqxx} and the same choice of $(\epsilon_i)_{i=1,2,\ldots,d}$ in Section \ref{subsec:minimax} with the argument $s(\sigma)$ replaced by $s^*(\sigma)$, we obtain minimax lower bound
\begin{align*}
    \inf_{\hat{\theta}} \sup_{\theta \in \mathbb{B}_2^d(1)} \E\|\theta - \hat{\theta}\|_2^2 \gtrsim \min\left( \frac{ s^*(\sigma)}{\LL}  \tr\left((\fixed^\top \fixed)^{-1}\right) , 1\right).
\end{align*}

In the special case where $s_1(\sigma) = \ldots = s_n(\sigma)$, the same minimax lower bound as Theorem \ref{thm:minimaxBound} is recovered.


\section*{Acknowledgements}
The authors thank Ashwin Pananjady for useful discussions.  This work was supported in part by NSF grants CCF-1704967, CCF-0939370 and CCF-1750430.

\bibliographystyle{IEEEtran}
\bibliography{refs}

\end{document}